\documentclass[11pt, oneside]{article}   	
\usepackage{amsmath, amsthm, amsfonts, amssymb}
\usepackage{color}
\usepackage{newtxtext,newtxmath}
\usepackage[]{hyperref}
\newcommand\Prob{{\mathbb P}}

\newcommand\Normal{{\mathcal N}}

\newcommand\X{{\mathbb X}}

\newcommand\Sp{{\mathcal S}}

\newcommand\bfb{{\boldsymbol b}}
\newcommand\bfx{{\boldsymbol x}}
\newcommand\bfy{{\boldsymbol y}}
\newcommand\bfz{{\boldsymbol z}}
\newcommand\bflambda{{\boldsymbol \lambda}}
\newcommand\bfphi{{\boldsymbol \phi}}

\newcommand\cA{{\mathcal C}}
\newcommand\cWA{{W_C}}

\newcommand\cF{{\mathcal F}}

\newcommand\cL{{\mathcal L}}

\newcommand\N{{\mathbb N}}

\newcommand\R{{\mathbb R}}

\newcommand\dto{\overset{d}{\to }}

\newcommand\Pto{\overset{\Prob}{\to}}

\newcommand\bra[1]{\langle #1 \rangle}
\newcommand\bralr[1]{\left\langle #1 \right\rangle}

\DeclareMathOperator{\Beta}{Beta}

\DeclareMathOperator{\tr}{tr}

\DeclareMathOperator{\Var}{Var}
\DeclareMathOperator{\Ex}{\mathbb E}
\DeclareMathOperator{\Cov}{Cov}

\DeclareMathOperator{\diag}{diag}

\newtheorem{theorem}{Theorem}[section]

\newtheorem{lemma}[theorem]{Lemma}
\newtheorem{proposition}[theorem]{Proposition}

\theoremstyle{definition}

\theoremstyle{remark}
\newtheorem{remark}[theorem]{Remark}

\title{Beta Jacobi ensembles and associated Jacobi polynomials, II}

\author{
Fumihiko Nakano\footnote{Mathematical Institute, Tohoku University, Sendai,  Japan.
\newline Email: fumihiko.nakano.e4@tohoku.ac.jp}
\and
Hoang Dung Trinh\footnote{Faculty of Mathematics Mechanics Informatics, University of Science, Vietnam National University, Hanoi, Vietnam.
\newline Email: thdung.hus@gmail.com} 
\and
Khanh Duy Trinh\footnote{Global Center for Science and Engineering, Waseda University, Japan.
\newline
Email: trinh@aoni.waseda.jp 
} 
}

\begin{document}

\maketitle
\begin{abstract}
In a high temperature regime, it was shown in Trinh--Trinh (\emph{J.\ Stat.\ Phys.}\ \textbf{185}(1), Paper No.\ 4, 15 (2021)) that the empirical distribution of beta Jacobi ensembles converges to a limiting probability measure which is related to Model III of associated Jacobi polynomials. In this paper, we establish Gaussian fluctuations around the limit whose statement involves orthogonal polynomials. For the proofs, we refine a moment method at the process level which has been used to deal with Gaussian beta ensembles and beta Laguerre ensembles.

		
\medskip
	
\end{abstract}

\section{Introduction}

Beta Jacobi ensembles are a family of probability distributions on $[0,1]^N$ with joint probability density functions of the form
\begin{equation}\label{bJE}
\frac 1 {Z} \times  \prod_{i < j} |\lambda_j - \lambda_i|^\beta \prod_{l = 1}^N \lambda_l^a (1 - \lambda_l)^b, \quad (0 \le \lambda_1\le \cdots \le \lambda_N \le 1),
\end{equation} 
where $N$ is the system size, $Z$ is the normalizing constant and $\beta > 0$ and $a,b > -1$ are parameters. They can be realized as the joint distribution of the eigenvalues of random tridiagonal matrices \cite{Killip-Nenciu-2004}. In a high temperature regime where $N \to \infty$ with $\beta N \to 2c \in (0, \infty)$, and $a, b$ are fixed, (with $c$ being a constant), the empirical distribution 
\[
	L_{N, \beta} = \frac1N \sum_{i=1}^N \delta_{\lambda_i}
\]
converges weakly to a probability measure $\nu_c$ (depending on $a$ and $b$ as well) which is the probability measure of a new model (Model III) of associated Jacobi polynomials, almost surely \cite{Trinh-Trinh-Jacobi}. Here $\delta_\lambda$ denotes the Dirac measure at $\lambda \in \R$. Since all probability measures $L_{N, \beta}$ are supported on $[0,1]$, the above convergence of random probability measures is equivalent to the following: for any polynomial test function $p$, as $N \to \infty$ with $\beta N \to 2c$, 
\[
	\bra{L_{N, \beta}, p} = \frac1N \sum_{i=1}^N p(\lambda_i) \to \bra{\nu_c, p}, \quad \text{almost surely.}
\] 
Here $\bra{\mu, f}$ denotes the integral $\int f d\mu$, for an integrable function $f$ with respect to a probability measure $\mu$. The above almost sure convergence is referred to as the law of large numbers (LLN) for polynomial test functions.

The aim of this paper is to study fluctuations around the limit. We establish central limit theorems (CLT) for $\{\bra{L_{N, \beta}, p}\}$ with polynomial test functions $p$. Moreover, we obtain multivariate CLTs involving orthogonal polynomials with respect to the measure $x(1-x) \nu_c(x)dx$. Let us introduce the result in detail. Denote by $\nu_c^*$ the probability measure with density $\nu_c^*(x)$ proportional to $x(1-x)\nu_c(x)$, that is,
\[
	\nu_c^*(x) = \frac1{Z_c^*} x(1-x) \nu_c(x),	\quad x \in (0,1), \quad (Z_c^* = \bra{\nu_c, x(1-x)}).
\]
Let $\{\tilde p_n\}_{n \ge 0}$ be orthonormal polynomials with respect to $\nu_c^*$ which are defined by the following three term recurrence relation (see Sect.~\ref{sect:Duals})
\begin{equation}\label{orthogonal-polynomials}
\begin{cases}
	\tilde  p_0(x) = 1, \quad b_1^* \tilde p_1(x) = x - a_1^*,\\
	b_{n+1}^* \tilde p_{n+1}(x) = x \tilde p_n(x)  - a_{n+1}^* \tilde p_n(x) - b_n^* \tilde p_{n-1}(x), \quad n = 1,2,\dots,
\end{cases}
\end{equation}
where $a_n^* = \lambda_{n-1}^* + \mu_{n-1}^*$ and $b_n^* = \sqrt{\mu_{n-1}^* \lambda_n^*}$, with
\[
\begin{cases}
	\lambda_n^*  = \frac{n+c+a+1}{2n+2c+a+b+2}\frac{n+c+a+b+2}{2n+2c+a+b+3},\\
	\mu_n^* = \frac{n+c+1}{2n+2c+a+b+3}\frac{n+c+b+2}{2n+2c+a+b+4}, 
\end{cases}
\quad n \ge 0.
\]
Let $\tilde P_n$ be a primitive of $\tilde p_n$, that is, $\tilde P_n' = \tilde p_n$. Then our main result on Gaussian fluctuations around the limit is stated as follows.

\begin{theorem}\label{thm:CLT-intro}
Let $a, b > -1/2$ and $c > 0$ be given. Then the following hold.
\begin{itemize}
	\item[\rm(i)] For $n \ge 0$, as $N \to \infty$ with $\beta = 2c/N$, 
	\begin{equation}\label{intro:CLT}
		\sqrt{N} \left(\bra{L_{N, \beta}, \tilde P_n} - \bra{\nu_c, \tilde P_n} \right) \dto \Normal(0, \sigma_{\tilde P_n}^2),\quad \sigma_{\tilde P_n}^2 =  \frac{Z_c^*}{(n+1)(n+2c+a+b+2) }.
	\end{equation}
Here `$\dto$' denotes the convergence in distribution, $\Normal(\mu, \sigma^2)$ denotes the Gaussian distribution with mean $\mu$ and variance $\sigma^2$, and
\[
	Z_c^* = \bra{\nu_c, x(1-x)} = \frac{(c+a+1) (c+b+1) (c+ a + b + 2)}{(2c + a + b + 2)^2 (2c+ a + b + 3)}.
\]

\item[\rm(ii)] For any $M \in \N$, $\big\{\sqrt{N}\big(\bra{L_{N, \beta}, \tilde P_n} - \bra{\nu_c, \tilde P_n}\big)\big\}_{n=0}^M$ jointly converge in distribution to independent Gaussian random variables.
\end{itemize}
\end{theorem}

Note that the condition $a, b > -1/2$ is assumed when dealing with beta Jacobi processes (see Sect.~3). The above result is expected to hold for $a, b > -1$. Note also that the convergence~\eqref{intro:CLT} holds under the condition that $\beta = 2c/N$. We still get an analogous result when $\beta N \to 2c$ by replacing $\bra{\nu_c, \tilde P_n}$ by $\Ex[\bra{L_{N, \beta}, \tilde P_n}] $, that is, as $\beta N \to 2c$,
\[
	\sqrt{N} \left(\bra{L_{N, \beta}, \tilde P_n} - \Ex[\bra{L_{N, \beta}, \tilde P_n}] \right) \dto \Normal(0, \sigma_{\tilde P_n}^2).
\]

We conclude the introduction with the following remark.
Theorem~\ref{thm:LLN-CLT} below implies that
for any $k, l \ge 1$, the limit  
\[
	\sigma_{k, l} = \lim_{N \to \infty; \beta = 2c/N} N \Cov [\bra{L_{N, \beta}, x^k}, \bra{L_{N, \beta}, x^l}]
\]
exists and that for any $M \ge 1$, 
\[
	\sqrt{N} \left(\bra{L_{N, \beta}, x^k} - \bra{\nu_c, x^k} \right)_{k=1}^M \dto \Normal_M(0, \Sigma), \quad \Sigma = (\sigma_{k,l})_{k,l=1}^M,
\]
where $\Normal_M(0, \Sigma)$ denotes the $M$-dimensional Gaussian distribution with mean zero and covariance matrix $\Sigma$. Besides, the limiting covariance $\sigma_{k,l}$ satisfies the following relation (Eq.~(5.23) in \cite{Forrester-2021})
\[
	\sigma_{k, l} = \frac{1}{k+2c+a+b+1} \left(l (u_l - u_{k+l}) - b \sum_{i=1}^{k-1} \sigma_{i, l} - 2c \sum_{i=1}^{k-1} u_i \sigma_{k-i, l} \right).
\]
Here $u_k = \bra{\nu_c, x^k}$ is the $k$th moment of the limiting measure $\nu_c$. There is also a recursive formula for $\{u_k\}$ as given in equation \eqref{moments-uk}. Now for each $n$, we can assume without loss of generality that the constant term in a polynomial  $\tilde P_n$ is zero. Given $M \ge 1$, let $C$ be an $M \times M$ coefficients matrix of $\{\tilde P_n\}_{n=0}^{M-1}$, that is,
\[
	\begin{pmatrix}
		\tilde P_0\\
		\tilde P_1\\
		\vdots\\
		\tilde P_{M-1}
	\end{pmatrix}
	= C 
	\begin{pmatrix}
		x\\
		x^2\\
		\vdots\\
		x^{M}
	\end{pmatrix}.
\]
It is clear that $C$ is a lower triangular matrix. Then Theorem~\ref{thm:CLT-intro}(ii) implies that
\begin{equation}
	\diag(\sigma_{\tilde P_0}^2, \dots, \sigma_{\tilde P_{M-1}}^2) = C \Sigma C^\top,
\end{equation}
where $\diag(a_1, \dots, a_M)$ denotes the diagonal matrix with $\{a_n\}_{n=1}^M$ in the diagonal, and $C^\top$ denotes the transpose of $C$.
It would be an interesting problem to study properties of the limiting variance matrix $\Sigma$ like the above decomposition which is related to  the LDL decomposition directly from the recursive formula for $\sigma_{k,l}$.

The paper is organized as follows. In Sect.\ 2, we establish the LLN and CLTs for polynomial test functions by using the random tridiagonal matrix model. Sect.\ 3 proves a dynamic version of  Theorem~\ref{thm:CLT-intro} for stationary beta Jacobi processes. Several properties of associated Jacobi polynomials needed in arguments in Sect.\ 3 are introduced in Sect.\ 4. Finally, the proof of Theorem~\ref{thm:CLT-intro} is given in Sect.\ 5.

\section{Limit theorems for the empirical distribution of beta Jacobi ensembles}
In this section, we establish the convergence to a limit and Gaussian fluctuations around the limit of the empirical distribution by using the tridiagonal random matrix model. We begin with introducing the model.
Denote by $\Beta(a,b)$ the beta distribution with parameters $a, b > 0$.
Let $p_1, \dots, p_N$ and $q_1, \dots, q_{N - 1}$ be independent random variables having beta distributions with parameters
\begin{align*}
	p_n &\sim \Beta((N - n) \tfrac{\beta}2 + a + 1, (N - n) \tfrac{\beta}2 + b + 1),\\
	q_n &\sim \Beta((N - n) \tfrac{\beta}2, (N - n - 1) \tfrac{\beta}2 + a + b + 2),
\end{align*}
where $a, b > -1$ and $\beta > 0$. 
Define
\begin{align*}
	s_n &= {p_n(1 - q_{n - 1})},\quad n = 1, \dots, N, \quad (q_0 = 0),\\
	t_n &= {q_n(1 - p_n)}, \quad n = 1, \dots, N - 1,
\end{align*}
and form a (random symmetric) tridiagonal matrix
\[
	J_{N, \beta} = \begin{pmatrix}
		\sqrt{s_1}	\\
		\sqrt{t_1}	&\sqrt{s_2}		\\
		&\ddots	&\ddots \\
		&& \sqrt{t_{N - 1}}	& \sqrt{s_N}
	\end{pmatrix}
	 \begin{pmatrix}
		\sqrt{s_1} 	&\sqrt{t_1}	\\
			&\sqrt{s_2}		&\sqrt{t_2}		\\
		&&\ddots	&\ddots \\
		&&& \sqrt{s_N}
	\end{pmatrix}.
\]
Then the eigenvalues $\lambda_1 \le \lambda_2 \le \cdots \le \lambda_N$ of $J_{N, \beta}$ are distributed according to the beta Jacobi ensemble~\eqref{bJE} \cite{Killip-Nenciu-2004}.

For given $a, b > -1$ and $c > 0$, we form the following random (infinite) tridiagonal matrix
\[
J^{(\infty)}_c = \begin{pmatrix}
		\sqrt{s_1^{(\infty)}}	\\
		\sqrt{t_1^{(\infty)}}	&\sqrt{s_2^{(\infty)}}		\\
		&\ddots	&\ddots \\
	\end{pmatrix}
	 \begin{pmatrix}
		\sqrt{s_1^{(\infty)}} 	&\sqrt{t_1^{(\infty)}}	\\
			&\sqrt{s_2^{(\infty)}}		&\sqrt{t_2^{(\infty)}}		\\
		&&\ddots	&\ddots \\
	\end{pmatrix},
\]
where 
\[
\begin{cases}
	s_n^{(\infty)} = {p_n^{(\infty)}(1 - q_{n - 1}^{(\infty)})},\\
	t_n^{(\infty)} = {q_n^{(\infty)}(1 - p_n^{(\infty)})}, 
\end{cases}
\quad n = 1,  2, \dots, \quad (q_0^{(\infty)} = 0),
\]
with $\{p_n^{(\infty)}\}$ and $\{q_n^{(\infty)}\}$ being two sequences of i.i.d.\ (independent identically distributed) random variables independent of each other with the common distributions $ \Beta(c + a + 1, c + b + 1)$ and $\Beta(c, c + a + b + 2)$, respectively. The matrix $J^{(\infty)}_c$ is the limit in distribution of $J_{N, \beta}$ as $N \to \infty$ with $N \beta \to 2c$ in the sense that for each fixed $i$, 
\[
	J_{N, \beta}(i, i) \dto J^{(\infty)}_c(i, i), \quad J_{N, \beta}(i, i\pm1) \dto J^{(\infty)}_c(i, i\pm1).
\]
The joint convergence also holds.

Let $\Sp_{N, \beta}$ be the spectral measure of $J_{N, \beta}$, the unique probability measure satisfying 
\[
	\bra{\Sp_{N, \beta}, x^n} = (J_{N, \beta})^n(1,1), \quad n = 0,1,\dots.
\]
Relations between the spectral measure of tridiagonal matrices and orthogonal polynomials will be discussed in Sect.\ \ref{sect:Duals}.
It turns out that $\Sp_{N, \beta}$ has the following expression
\[
	\Sp_{N, \beta} = \sum_{i=1}^N w_i \delta_{\lambda_i},
\]
where $w_i = v_i(1)^2$ with $\{v_i\}_{i=1}^N$ normalized eigenvectors corresponding to the eigenvalues $\{\lambda_i\}_{i=1}^N$ of $J_{N, \beta}$. Moreover, it is known that $(w_i)_{i=1}^N$ are independent of the eigenvalues and have the Dirichlet distribution with parameters $(\beta/2, \dots, \beta/2)$ (see \cite{Killip-Nenciu-2004} or Chapter 3 in \cite{Forrester-book}). Consequently, for any test function $f$, 
\[
	\Ex[\bra{\Sp_{N, \beta}, f}] = \Ex \bigg[ \sum_{i=1}^N w_i f(\lambda_i)\bigg] = \sum_{i=1}^N \Ex[w_i] \Ex[f(\lambda_i)] = \frac1N \Ex \bigg[\sum_{i=1}^N f(\lambda_i) \bigg] = \Ex[\bra{L_{N, \beta}, f}].
\]
In particular, when $f = x^k$, the above relation implies that
\begin{equation}\label{same-mean}
	\Ex[\bra{L_{N, \beta}, x^k}] = \Ex \bigg[\frac1N \sum_{i=1}^N (J_{N, \beta})^k (i, i)\bigg] = \Ex[(J_{N,\beta})^k(1,1)].
\end{equation}

Let $\Sp_c$ be the spectral measure of $J^{(\infty)}_c$, the unique probability measure satisfying
\[
	\bra{\Sp_c, x^n} = (J^{(\infty)}_c)^n(1,1), \quad n = 0,1,\dots.
\]
It is unique because entries of $J^{(\infty)}_c$ are bounded (see Sect.~\ref{sect:Duals}). Let $\nu_c$ be the mean of $\Sp_c$. The convergence in distribution of $\{p_n\}$, $\{q_n\}$ together with their independence implies that as $\beta N \to 2c$,
\begin{equation}\label{convergence-of-mean}
	\Ex[(J_{N, \beta})^k(1,1)] \to \Ex[(J^{(\infty)}_c)^k(1,1)] = \bra{\nu_c, x^k}.
\end{equation}
When $\beta = 2c/N$, we estimate the rate of convergence as follows.
\begin{lemma}\label{lem:rate-of-mean}
For any $k \in \N$, in the regime where $N \to \infty$ with $\beta = 2c/N$, there is a constant $C>0$ such that
\[
	|\Ex[(J_{N, \beta})^k(1,1)] - \Ex[(J^{(\infty)}_c)^k(1,1)] | \le \frac{C}{N}.
\]
Consequently,
for any polynomial $p$, as $N \to \infty$ with $\beta = 2c/N$, 
\begin{equation}\label{convergence-of-mean-with-rate}
	\sqrt{N} \big(\Ex[\bra{L_{N, \beta}, p}] - \bra{\nu_c, p} \big) \to 0.
\end{equation}
\end{lemma}
We need the following result to prove the above lemma.

\begin{lemma}\label{lem:beta-distributions}
\begin{itemize}
\item[\rm (i)]	Assume that $X \sim \Beta(a, b)$ has the beta distribution with parameters $a, b > 0$. Then for any $k \in \N$,
\[
	\Ex[X^k] = \prod_{r=1}^k \frac{a+r}{a+b+r}.
\]

\item[\rm(ii)] For given $a, b > 0, 0 <\varepsilon < a \wedge b$ and $K \in \{1,2,\dots\}$, there exists a constant $C > 0$ such that for any $Y \sim \Beta(a+\delta_a, b + \delta_b)$ with $|\delta_a|, |\delta_b| < \varepsilon$, and for any $k \in \{1,2,\dots,K\}$,
\[
	|\Ex[Y^k] - \Ex[X^k]| \le C \varepsilon.
\]
Here $a \wedge b = \min\{a, b\}$.
\end{itemize}
\end{lemma}
\begin{proof}
	(i) is a fundamental result on beta distributions. (ii) is deduced from (i) by using the mean value theorem for functions of two variables. 
\end{proof}

\begin{proof}[Proof of Lemma~\rm\ref{lem:rate-of-mean}]
For given $k \in \N$, it is clear that when $N \ge k$, $(J_{N, \beta})^k(1,1)$ is a polynomial of $\{s_i, t_i\}_{1\le i \le \ell_k}$, and thus, is a polynomial of $\{p_i, q_i\}_{1\le i \le \ell_k}$, where $\ell_k$ depends only on $k$. A general term in that polynomial is of the form
\[
	const \times \prod_{i=1}^{\ell_k} p_i^{a_i} q_i^{b_i},
\]
with $a_i, b_i \in \{0,1,2,\dots\}$ and $\sum_i (a_i + b_i) \le 2k$. By the independence of $\{p_i\}_i$ and $\{q_i\}_i$, its mean value is a product of moments of beta distributions.
Now, we make use of Lemma~\ref{lem:beta-distributions}(ii) to deduce that in the regime where $ \beta = 2c/N$,
\[
	|\Ex[(J_{N, \beta})^k(1,1)] - \Ex[(J^{(\infty)}_c)^k(1,1)] | \le \frac{C}{N}, 
\]
for some constant $C>0$ depending on $k, a, b$ and $c$. The convergence~\eqref{convergence-of-mean-with-rate} is a direct consequence of the above estimate. The proof is complete.
\end{proof}

Here is the main result in this section.
\begin{theorem}\label{thm:LLN-CLT}
Let $a, b > -1$ and $c > 0$ be fixed. Then the following hold.
\begin{itemize}
\item[\rm(LLN)]
As $N \to \infty$ with $\beta N \to 2c$, the empirical distribution $L_{N, \beta}$ converges weakly to $\nu_c$, almost surely. Equivalently, for any polynomial $p$, as $N \to \infty$ with $\beta N \to 2c$,
\begin{equation}
	\bra{L_{N, \beta}, p} \to \bra{\nu_c, p}, \quad a.s..
\end{equation}

\item[\rm(CLT)]
For any polynomial $p$, as $N \to \infty$ with $\beta = 2c/N$, 
\begin{equation}
	N \Var[\bra{L_{N, \beta}, p}] \to \sigma_p^2, \quad \sqrt{N}\big(\bra{L_{N, \beta}, p} - \bra{\nu_c, p}\big) \dto \Normal(0, \sigma_p^2),
\end{equation}
where $\sigma_p^2 \ge 0$ is a constant. Equivalently, there are jointly Gaussian random variables $\{\zeta_n\}_{n \ge 1}$ such that for any $M \in \N$, 
\begin{equation}
	\left (\sqrt{N}(\bra{L_{N, \beta}, x^n} - \bra{\nu_c, x^n})	\right )_{n=1}^M \dto (\zeta_n)_{n=1}^M \sim \Normal_M(0, \Sigma),
\end{equation}
where $\Sigma = (\sigma_{k, l})_{k, l = 1}^M$ is the limiting covariance matrix, 
\[
	\sigma_{k, l} = \lim_{N \to \infty; \beta = 2c/N} N \Cov [\bra{L_{N, \beta}, x^k}, \bra{L_{N, \beta}, x^l}].
\]
\end{itemize}
\end{theorem}
\begin{proof}
(LLN) For any polynomial $p$, the convergence of the mean value of $\bra{L_{N, \beta}, p}$ follows directly from equation~\eqref{convergence-of-mean}. The almost sure convergence of the form 
\[
	\bra{L_{N, \beta}, p} - \Ex[\bra{L_{N, \beta}, p}] \to 0, \quad a.s.,
\]
can be deduced by using the random matrix model $J_{N, \beta}$ (see \cite{Trinh-Trinh-Jacobi}). Thus, we get the first part proved.

(CLT) We only present key ideas because arguments are almost the same as those used in \cite{Nakano-Trinh-2018} for the Gaussian case. First, it is clear that $\bra{L_{N, \beta}, p}$ is a polynomial of $\{p_n, q_n\}_{n=1}^N$. If the two sequences $\{p_n\}$ and $\{q_n\}$ were i.i.d.\ and independent of each other, an application of the martingale difference central limit theorem with a filtration constructed from $p_n$'s and $q_n$'s leads to the following CLT (cf.\ \cite[\S 2.1]{Nakano-Trinh-2018})
\begin{equation}\label{CLTiid}
	N \Var[\bra{L_{N, \beta}, p}] \to \sigma_p^2, \quad \sqrt{N}\big(\bra{L_{N, \beta}, p} - \Ex[\bra{L_{N, \beta}, p}]\big) \dto \Normal(0, \sigma_p^2),\quad (\sigma_p^2 \ge 0).
\end{equation}
Second, in the regime where $\beta = 2c/N$, the two sequences $\{p_n\}$ and $\{q_n\}$ are locally approximated by i.i.d.\ sequences, and thus, arguments as used in \cite[\S 2.2]{Nakano-Trinh-2018} for the Gaussian case also work here to yield the CLT~\eqref{CLTiid}. Third, the mean value $\Ex[\bra{L_{N,\beta}, p}]$ can by replaced by $\bra{\nu_c, p}$ in \eqref{CLTiid} by Lemma~\ref{lem:rate-of-mean}. Finally, the equivalent statement for moments is standard, which completes the proof of Theorem~\ref{thm:LLN-CLT}.
\end{proof}

\section{Limit theorems for the empirical measure process of beta Jacobi processes}

\subsection{Beta Jacobi processes}
Let $a, b > -1/2$ and $\beta > 0$. Consider the following system of stochastic differential equations (SDE)
\begin{align}
	d\phi_i 
	&={\sqrt 2} db_i + \frac{2a + 1}{2}\cot \frac{\phi_i}2 dt - \frac{2b+1}{2} \tan \frac{\phi_i}{2} dt \notag\\
	&\quad+ \frac\beta 2 \sum_{j: j \neq i}^N \left(\cot(\frac{\phi_i}2 + \frac{\phi_j}2) + \cot(\frac{\phi_i}2 - \frac{\phi_j}2)\right) dt, \quad(i = 1, 2, \dots, N),	\label{SDEs-phi}
\end{align}
with initial conditions $\phi_i(0) = y_i, (i = 1, 2, \dots, N)$, where $\bfy = (y_1,  \dots, y_N) \in \bar \cA$, the closure of
\[
	\cA= \left\{\bfx = (x_1, \dots, x_N) \in \R^N : 0 < x_1 < \cdots < x_N < {\pi} 	\right\}.
\]
Here $\{b_i\}_{i = 1}^N = \{b_i(t)\}_{i=1}^N$ are independent standard Brownian motions. The system of SDEs~\eqref{SDEs-phi} has a unique strong solution under the constraint that $(\phi_1(t), \dots, \phi_N(t)) \in \bar \cA$ for all $t \in [0, \infty)$ c.f.~\cite{Demni-2010}. To deal with the singularity of the drift term at the boundary of $\cA$, the theory of multivariate SDEs developed in \cite{Cepa-1995, Cepa-Lepingle-1997} has been used. 

Define a function $\Psi \colon \R^N \to (-\infty, \infty]$ as
\begin{align*}
	\Psi(\bfx) &= (2a + 1) \sum_i \Big( -  \log \sin \frac{x_i}2\Big) + (2b + 1) \sum_i  \Big(-\log \cos \frac{x_i}2\Big) \\
	&\quad + \beta \sum_{i < j} \Big\{- \log \sin \Big(\frac{x_i}2 + \frac{x_j}2 \Big) -  \log \sin \Big(\frac{x_j}2 - \frac{x_i}2 \Big) \Big\}, \quad \bfx = (x_1, \dots, x_N) \in \cA,
\end{align*}
and $\Psi(\bfx) = \infty$, otherwise. Note that under the condition $a, b > -1/2$ and $\beta > 0$, the function $\Psi$ is convex with the domain $\cA$. For $\bfx \in \cA$, it is clear that 
\[
	\frac{\partial \Psi}{\partial x_i}(\bfx) = - \frac{2a + 1}{2}  \cot \frac{x_i}2 +  \frac{2b + 1}{2} \tan \frac{x_i}2 - \frac \beta 2 \sum_{j \neq i} \left \{\cot \Big(\frac{x_i}2 + \frac{x_j}2 \Big) + \cot \Big(\frac{x_i}2 - \frac{x_j}2 \Big)\right \},
\]
and thus, the system of SDEs~\eqref{SDEs-phi} for $\bfphi_t = (\phi_i(t))_{i=1}^N$ can be expressed as
\begin{equation}\label{SDEs-phi-grad}
	d\bfphi_t = {\sqrt 2} d\bfb_t - \nabla \Psi(\bfphi_t) dt.
\end{equation}
Here $\bfb_t = (b_i(t))_{i=1}^N$ is a standard Brownian motion in $\R^N$, and $\nabla \Psi = (\frac{\partial \Psi}{\partial x_i})_{i=1}^N$. It is worth mentioning that the above expression makes it easy to use the theory of multivariate SDEs since the subderivative of a convex function is a maximal monotone operator.

Beta Jacobi processes are then defined via a change of variables $\lambda_i = (\sin \frac{\phi_i}2)^2, i = 1, \dots, N$. They satisfy the following system of SDEs
\begin{equation}\label{SDEs-J}
\begin{cases}
	d \lambda_i =  \sqrt{2\lambda_i (1 - \lambda_i)} db_i + \left(a + 1 - (a+b+2) \lambda_i + \frac\beta2\sum_{j: j \neq i} \frac{2 \lambda_i(1 - \lambda_i)} {\lambda_i - \lambda_j} \right) dt, \\
	\lambda_i(0) = x_i, 
\end{cases}
\quad i=1,\dots, N.
\end{equation}
Recall that $a, b > -1/2$ and $\beta > 0$.

For any $\varphi(t, \bfx)  \in C_0^{2} ((0, \infty) \times \cA)$, the space of continuous functions with continuous derivatives up to second order and with compact support, using It\^o's formula, we get that for any $T > 0$,
\begin{align*}
	f(T, \bfphi_T) &= f(0, \bfphi_0) + \sum_{i=1}^N \int_0^T \partial_i f(t, \bfphi_t)\sqrt 2 db_i(t) \\
	&\quad + \int_0^T \left( \partial_t f(t, \bfphi_t)  - \sum_{i=1}^N \partial_i f(t, \bfphi_t) \partial_i \Psi(\bfphi_t) + \sum_{i=1}^N \partial_i^2 f(t, \bfphi_t)\right) dt.
\end{align*}
Here $\partial_t$ (resp.\ $\partial_i$) denotes the partial derivative with respect to $t$ (resp.\ $x_i$).
Since $f$ has compact support, $f(0, \bfphi_0) = 0$ and $f(T, \bfphi_T) = 0$ when $T$ is large enough. Taking the mean values of both sides in the above formula with noting that the diffusion part is a martingale, we get that
\[
	\int_0^\infty \int_\cA \left( \partial_t f(t, \bfz)  - \sum_{i=1}^N \partial_i f(t, \bfz) \partial_i \Psi(\bfz) + \sum_{i=1}^N \partial_i^2 f(t, \bfz)\right) d\nu_t^{[\bfy]} (d\bfz) dt = 0,
\]
where $\nu_t^{[\bfy]}$ is the law of the unique strong solution $\bfphi_t$ with initial condition $\bfphi_0 = \bfy \in \bar \cA$. In other words, the law $\nu_t^{[\bfy]}$ satisfies the following Fokker--Planck equation 
\begin{equation}\label{FP}
	\frac{\partial \mu}{\partial t} = \mathrm{div\,} (\nabla \Psi \mu) + \Delta\, \mu, \quad \mu = (\mu_t(dx))_{t > 0}
\end{equation}
in a distributional sense with initial condition $\delta_{\bfy}$. 

The function $\Psi$ is in fact $K$-convex with 
\[
	K = \frac{a + b + 1}{2} >0,
\]
that is, the function $\Psi(x_1, \dots, x_N) - \frac K 2(x_1^2 + \cdots + x_N^2)$ is convex, which can be verified by a direct calculation.
Now, from the theory of gradient flows, we deduce that $\nu_t^{[\bfy]}$ is the gradient flow of a $K$-geodesically convex functional in the space of probability measures on $\R^N$ with finite second moment endowed with the Wasserstein distance $W_2$. Refer to the monograph \cite{AGS-book}, especially to Subsection 11.2.1 therein for terminologies. Moreover, the probability measure $\mu_{inv}$ on $\cA$ or on $\bar \cA$ with density proportional to $e^{-\Psi(\bfx)}$ is the unique invariant measure of the Fokker--Planck equation~\eqref{FP} or of the system of SDEs~\eqref{SDEs-phi}. Then the $K$-convexity implies the following long time behavior (see \cite[\S11.2]{AGS-book})
\begin{equation}\label{equilibrium}
	W_2(\nu_t^{[\bfy]}, \mu_{inv}) \le e^{-K t} W_2(\delta_{\bfy}, \mu_{inv}).
\end{equation}
Note that the push-forward measure of $\mu_{inv}$ under the transformation $\bflambda = \sin^2 \frac{\bfphi}2$, is the beta Jacobi ensemble~\eqref{bJE}. Thus, a beta Jacobi process, starting from any initial condition, converges in distribution to the corresponding beta Jacobi ensemble as $t \to \infty$.

\subsection{Law of large numbers for the empirical measure process}
From now on, we consider stationary beta Jacobi processes. More precisely, we consider the processes $\{\lambda_i\}_{i=1}^N$ solving the system of SDEs~\eqref{SDEs-J} with initial data $(x_1, \dots, x_N)$ being random, $\cF_0$ measurable, distributed according the Jacobi beta ensemble~\eqref{bJE}, where $\bfb_t = (b_i(t))_{i=1}^N$ is a standard Brownian motion in $\R^N$ with respect to (w.r.t.) the filtration $\{\cF_t\}_{t\ge 0}$. In other words, the Brownian motion is independent of the initial data. 

The law of large numbers for the empirical measure process
\[
	\mu_t^{(N)} = \frac1N \sum_{i = 1}^N \delta_{\lambda_i(t)}
\]
has been established in \cite{Trinh-Trinh-Jacobi} by using a moment method for processes. 
Let us briefly introduce the result here. Recall the assumption on the parameters that $a, b > -1/2,c > 0$ and $\beta = 2c/N$. For $f \in C^2([0,1])$, using It\^o's formula, we deduce that
\begin{align*}
	d \bra{\mu_t^{(N)}, f} &= \frac1N \sum_{i = 1}^N \sqrt{2 \lambda_i(1 - \lambda_i)} f'(\lambda_i) db_i \\
	&+  \frac1N \sum_{i=1}^N f'(\lambda_i) \left( a + 1 - (a + b + 2) \lambda_i + \frac c N \sum_{j : j \neq i} \frac{2\lambda_i(1 - \lambda_i)}{\lambda_i - \lambda_j} \right) dt \\
	&+ \frac1N \sum_{i = 1}^N \lambda_i(1- \lambda_i) f''(\lambda_i) dt,
\end{align*}
which can be rewritten as 
\begin{align}
	d \bra{\mu_t^{(N)}, f} &= dM_t^{(N)}  +  \bra{\mu_t^{(N)}, (a + 1) f'(x) - (a + b + 2) x f'(x) + x(1-x) f''(x)} dt \notag\\
	&\quad + c \iint \frac{x(1-x) f'(x) - y(1-y) f'(y)}{x - y} d\mu_t^{(N)}(x) d\mu_t^{(N)}(y) dt  \notag\\
	&\quad - \frac cN \bra{\mu_t^{(N)}, \{x(1-x)f'(x)\}'} dt, \label{Ito-f}
\end{align}
where 
\[
	M_t^{(N)} = \frac1N \sum_{i = 1}^N \int_0^t \sqrt{2 \lambda_i(1 - \lambda_i)} f'(\lambda_i) db_i 
\]
is a martingale with the quadratic variation 
\[
	[M^{(N)} ]_t = \frac1N\int_0^t \bra{\mu_s^{(N)}, 2 x (1-x) f'(x)^2} ds.
\]
It is worth noting that the above identity holds when $\{\lambda_i(t)\}_{i=1}^N$ are distinct, which holds for almost every $t$.

For $k = 1,2, \dots,$ let 
\[	
	S_k^{(N)}(t) = \bra{\mu_t^{(N)}, x^k} = \frac1N \sum_{i = 1}^{N}  \lambda_i(t)^k
\] 
be the $k$th moment process of $\mu_t^{(N)}$. Then equation \eqref{Ito-f} with $f = x^k$ can be expressed in the following integral form 
\begin{align}
	S_k^{(N)} (t) &= \bra{\mu_0^{(N)}, x^k} + M_k^{(N)}(t)  - k (2c + a + b + k + 1) \int_0^t S_{k}^{(N)} (s)ds \notag \\
	&\quad +   k (a + k ) \int_0^t S_{k - 1}^{(N)}(s) ds \notag\\
	&\quad + c k \int_0^t \sum_{i = 0}^{k - 1} S_{i}^{(N)}(s) S_{k - 1 - i}^{(N)}(s) ds -  c k \int_0^t \sum_{j = 1}^{k-1} S_{j}^{(N)}(s) S_{k - j}^{(N)}(s) ds \notag\\
	&\quad - \frac cN \int_0^t (k^2 S_{k-1}^{(N)}(s) - k(k+1)S_k^{(N)}(s))ds \label{Ito-sk1}.
\end{align}
Here $M_k^{(N)}$ denotes the corresponding martingale part. Equation \eqref{Ito-sk1} is an initial value ordinary differential equation (ODE) of $\int_0^t S_k^{(N)}(s) ds$, and thus, we can express $S_k^{(N)}(t)$ as a functional of $\{S_l^{(N)}(t) \}_{0 \le l \le k -1}$ and $M_k^{(N)}(t)$ by using the following fundamental result stated without proof. 
\begin{lemma}\label{lem:ODE}
	Assume that a continuous function $X(t)$ satisfies the equation
\[
	X(t) =x_0 -\gamma \int_0^t X(s) ds + F(t),\quad (t \ge 0),
\]
where $x_0 \in \R$, $\gamma$ is a constant and $F(t)$ is a continuous function with $F(0) = 0$. Then 
\[
	X(t) =x_0 e^{-\gamma t} + F(t) - \gamma e^{-\gamma t} \int_0^t e^{\gamma s} F(s) ds, \quad (t \ge 0).
\]
\end{lemma}

For fixed $T > 0$, let $\X$ be the space $C([0, T]) = C([0, T], \R)$ of continuous functions on $[0, T]$ endowed with the supremum norm. Then $S_k^{(N)}$ and $M_k^{(N)}$ are random elements on $\X$. We can use induction in $k$ to show that $S_k^{(N)}$ converges in probability to a deterministic limit as $\X$-valued random elements. Indeed, by Doob's martingale inequality, the martingale part $M_k^{(N)}$ is easily shown to converge in probability to zero in $\X$. The initial value $\bra{\mu_0^{(N)}, x^k}$ converges in probability to $\bra{\nu_c, x^k}$ by Theorem~\ref{thm:LLN-CLT}(LLN).  Then, by induction, we deduce that $(S_k^{(N)}(t))_{0 \le t \le T}$ converges in probability to a deterministic limit $m_k(t)$ in $\X$. By this approach, the limit $m_k(t)$, for $k \ge 1$, satisfies the following ODE 
\begin{align*}
	m_k'(t) &= -k (2c + a + b + k + 1) m_k(t) + k (a + k ) m_{k - 1}(t) \notag\\
	&\quad + ck \sum_{i = 0}^{k-1} m_i(t) m_{k - 1 - i}(t) - ck \sum_{j = 1}^{k-1} m_j(t) m_{k  - j}(t), 
\end{align*}
with initial value $m_k(0) =  \lim_{N \to \infty} \bra{\mu_0^{(N)}, x^k} = \bra{\nu_c, x^k}.  $
Here $m_0(t) \equiv 1$.

Recall that $u_k = \bra{\nu_c, x^k}$ denotes the $k$th moment of $\nu_c$. Note that by the stationary property, for any fixed $t$, $S_k^{(N)}(t)$ converges to $u_k$ in probability (by Theorem~\ref{thm:LLN-CLT}(LLN)). Therefore, the limit $m_k(t)\equiv u_k$ is a constant function. Consequently, the moments $\{u_k\}$ satisfy the following relation
\begin{equation}\label{moments-uk}
 u_k =  \frac{1}{2c + a + b + k + 1} \bigg( (a + k ) u_{k - 1}
	 + c	 \sum_{i = 0}^{k-1} u_i  u_{k - 1 - i} - c \sum_{j = 1}^{k-1} u_j u_{k  - j} \bigg), \quad (k \ge 1).
\end{equation}
In summary, for stationary beta Jacobi point processes, the following law of large numbers at the process level holds.
\begin{theorem}\label{thm:LLN-processes}
As $N \to \infty$ with $\beta = 2c/N$, the $k$th moment process $S_k^{(N)}$, as random element in $ 
C([0, T])$, converges in probability to a constant function $u_k = \bra{\nu_c, x^k}$. Consequently, for any polynomial $f$, as $N \to \infty$ with $\beta = 2c/N$,
\begin{equation}\label{LLN-processes}
	\bra{\mu_t^{(N)}, f} \Pto \bra{\nu_c, f}	\quad \text{in}\quad C([0, T]),
\end{equation}
meaning that for any $\varepsilon > 0$, 
\[
	\lim_{\substack{N \to \infty;\\ \beta = 2c/N}} \Prob \left(\sup_{0 \le t \le T}|\bra{\mu_t^{(N)}, f} - \bra{\nu_c, f}| \ge \varepsilon \right) = 0.
\]
\end{theorem}

\subsection{Central limit theorem for moment processes}
In this subsection, we refine a moment method developed in \cite{NTT-2023}. The work in \cite{NTT-2023} deals with systems of SDEs with trivial initial condition while we consider here stationary processes.
Let 
\[
	\tilde S_k^{(N)}(t) = \sqrt{N} \left(S_k^{(N)}(t) - u_k \right).
\]
Then for each $k$, we deduce from equation~\eqref{Ito-sk1} that
\begin{align}
	\tilde S_k^{(N)} (t) &= \sqrt{N} (\bra{\mu_0^{(N)}, x^k} - u_k) + \sqrt N M_k^{(N)}(t)  - k (2c + a + b + k + 1) \int_0^t \tilde S_{k}^{(N)} (s)ds \notag \\
	&\quad +   k (a + k ) \int_0^t \tilde S_{k - 1}^{(N)}(s) ds \notag\\
	&\quad + c k \int_0^t \sum_{i = 0}^{k - 1} \Big(\sqrt N S_{i}^{(N)}(s) S_{k - 1 - i}^{(N)}(s)  - \sqrt N u_i u_{k-1-i} \Big)ds \notag \\
	&\quad -  c k \int_0^t \sum_{j = 1}^{k-1}\Big(\sqrt N S_{j}^{(N)}(s) S_{k - j}^{(N)}(s)  - \sqrt N u_j u_{k-j} \Big) ds \notag\\
	&\quad - \frac c{\sqrt N} \int_0^t (k^2 S_{k-1}^{(N)}(s) - k(k+1)S_k^{(N)}(s))ds \label{Ito-sk}.
\end{align}
Here we have used the relation~\eqref{moments-uk} for $\{u_k\}$. 
Although the last term depends on $S_k^{(N)}$, it is considered as an error term because it vanishes in the limit when $N \to \infty$.

Note that $\mu_0^{(N)}$ has the same distribution as the empirical distribution of the corresponding Jacobi beta ensemble. Thus Theorem~\ref{thm:LLN-CLT}(CLT) tells us that $\{\sqrt{N}(\bra{\mu_0^{(N)}, x^k} - u_k)\}_{k\ge1}$ jointly converge in distribution to jointly Gaussian random variables $\{\zeta_k\}_{k \ge 1}$, meaning that for every $M \in \N$, 
\[
	\big\{\sqrt{N}\big(\bra{\mu_0^{(N)}, x^k} - \bra{\nu_c, x^k}\big)\big\}_{k=1}^M \dto \{\zeta_k\}_{k=1}^M.
\]
We then define the limit with respect to any polynomial test function as a linear combination of $\{\zeta_k\}$, that is, for $f = \sum_{k=0}^M a_k x^k$,
\[
	\sqrt{N}(\bra{\mu_0^{(N)}, f} - \bra{\nu_c, f}) \dto \sum_{k=1}^M a_k \zeta_k =: \bra{\zeta, f}.
\]

Our next aim is to show that the martingale parts $\{\sqrt N M_k\}_{k \ge 1}$, as random elements in $C([0, T])$, jointly converge in distribution to Gaussian processes $\{\eta_k\}_{k \ge 1}$ chosen to be independent of $\{\zeta_k\}_{k \ge 1}$. Moreover, the initial values and the martingale parts jointly converge to $\zeta$'s and $\eta$'s. We use the same notation $\dto$ to denote the convergence in distribution of $C([0, T])$-valued random elements.

For a polynomial $f$, let 
\[
	\Phi^{(f; N)}(t) = \frac1 {\sqrt{N}} \sum_{i = 1}^N \int_0^t \sqrt{2 \lambda_i(s)(1 - \lambda_i(s))} f'(\lambda_i(s)) db_i(s).
\]
Then $\Phi^{(f; N)}$ is a continuous martingale with $\Phi^{(f; N)}(0) = 0$, a.s.. It is clear that for polynomials $f$ and $g$, the cross-variation of $\Phi^{(f; N)}(t)$ and $\Phi^{(g; N)}(t)$ is given by 
\begin{align*}
	[\Phi^{(f; N)}, \Phi^{(g; N)}]_t &= \frac1N \sum_{i=1}^N \int_0^t 2 \lambda_i(s)(1 - \lambda_i(s)) f'(\lambda_i(s)) g'(\lambda_i(s)) ds \\
	&= \int_0^t \bra{\mu^{(N)}_s, 2x(1-x)f'(x)g'(x)}ds.
\end{align*}

\begin{lemma}\label{lem:joint-convergence-polynomials}
Let $f_1, f_2, \dots, f_n$ be polynomials of degree at least $1$. Then $\{\Phi^{(f_k; N)}\}_{k=1}^n$ jointly converge to Gaussian processes $\{\eta_k\}_{k=1}^n$ of mean zero and  covariance 
	\[
		\Ex[\eta_k(s) \eta_l(t)] = \bra{\nu_{c}, 2x(1-x)f_k'(x) f_l'(x)} \times (s \wedge t).
	\] 
Here recall that $s \wedge t = \min\{s, t\}$.
\end{lemma}

\begin{proof}For any $k$ and $l$, as $N \to \infty$, Theorem~\ref{thm:LLN-processes} implies that 
\[
	[\Phi^{(f_k; N)}, \Phi^{(f_l; N)}]_t \Pto \bra{\nu_c, 2x(1-x)f_k'(x) f_l'(x)} \times t, \quad \text{in}\quad C([0, T]). 
\]
In particular, the convergence in probability holds for each $t > 0$. Moreover, for each fixed $t$, the sequence $[\Phi^{(f; N)}, \Phi^{(g; N)}]_t$ is uniformly bounded because $\mu_t^{(N)}$ is supported in $[0,1]$. It follows that $[\Phi^{(f_k; N)}, \Phi^{(f_l; N)}]_t$ converges in $L^q$ for any $q \in [1, \infty)$. Then the desired convergence is a consequence of a general CLT in \cite{Rebolledo-1980} for local martingales. The proof is complete.
\end{proof}

Let $\{\bra{\eta, f}\}$ be a family of Gaussian processes indexed by polynomials $f$ of mean zero and covariance 
\begin{equation}\label{covariance-fg}
	\Ex[\bra{\eta, f}(s) \bra{\eta, g}(t)] = \bra{\nu_{c}, 2x(1-x) f'(x) g'(x)}\times(s \wedge t),
\end{equation}
which is independent of $\{\bra{\zeta, f}\}$. In particular, for each polynomial $f$, the Gaussian process $\bra{\eta, f} / \sqrt{\bra{\nu_c, 2x(1-x)f'(x)^2}}$ is a standard Brownian motion.

\begin{lemma}\label{lem:joint-convergence-with-initial}
Let $f_1, f_2, \dots, f_n$ be polynomials of degree at least $1$. Then 
\[
	\Big\{\{\sqrt{N}(\bra{\mu_0^{(N)}, f_k} - \bra{\nu_c, f_k} ) \}_{k=1}^n, \{\Phi^{(f_k; N)}\}_{k=1}^n \Big\} \dto \Big\{ \{\bra{\zeta, f_k}\}_{k=1}^n, \{\bra{\eta, f_k}\}_{k=1}^n \Big\}.
\]
\end{lemma}

We need some preliminaries before proving this lemma. Let $M_t$ be a continuous martingale w.r.t.\ the filtration $\{\cF_t; 0 \le t < \infty\}$ satisfying $M_0 = 0$, a.s.,\ and $[M]_\infty = \infty$, a.s.. Let 
\[
	T_t = \inf \{s : [M]_s > t\}.
\] 
It is known that (see Theorem V.1.6 in \cite{Revuz-Yor-book})
\[
	B_t = M_{T_t}
\]
is a standard Brownian motion w.r.t.\ the filtration $\{\cF_{T_t}; 0 \le t < \infty\}$ and 
\[
	M_t = B_{[M]_t}.
\]
In particular, the Brownian motion $B_t$ is independent of $\cF_0$.
\begin{lemma}\label{lem:time-changed}
For each $N$, assume that $M_t^{(N)}$ is a continuous martingales with $M_t^{(N)} = 0,$ a.s.\ and $[M^{(N)}]_\infty = \infty,$ a.s.. Assume further that as $N \to \infty$,
\[
	[M^{(N)}]_t \Pto m_t,\quad \text{in} \quad C([0, T]),
\]
where $m_t \in C([0, T])$ is non-random.
Let $B^{(N)}_t$ be the standard Brownian motion associated with the martingale $M^{(N)}_t$ defined as above. Define a process $A^{(N)}_t$ as
\[
	A^{(N)}_t = B^{(N)}_{m_t}, \quad t \in [0, T].
\]
Then as $N \to \infty$,
\[
	M^{(N)}_t - A^{(N)}_t \Pto 0, \quad \text{in} \quad C([0, T]).
\]
\end{lemma}
\begin{proof}
For given $\varepsilon > 0$, we estimate the probability 
\[
	\Prob \left(\sup_{0 \le t \le T} |M^{(N)}_t - A^{(N)}_t| \ge \varepsilon\right)
\]
as follows. For any $\delta > 0$, 
\begin{align}
	\Prob\left(\sup_{0 \le t \le T} |M^{(N)}_t - A^{(N)}_t| \ge \varepsilon \right) &= 	\Prob\left(\sup_{0 \le t \le T} |B^{(N)}_{[M^{(N)}]_t} - B^{(N)}_{m_t}| \ge \varepsilon \right) \notag\\
	&\le \Prob\left(\sup_{0 \le t \le T} |B^{(N)}_{[M^{(N)}]_t} - B^{(N)}_{m_t}| \ge \varepsilon, \sup_{0 \le t \le T} |[M^{(N)}]_t - m_t| < \delta \right) \notag\\
	&\quad + \Prob\left( \sup_{0 \le t \le T} |[M^{(N)}]_t - m_t| \ge \delta\right)\notag\\
	&\le \Prob\left(\sup_{0 \le s \le m_T} \sup_{s' \ge 0, |s - s'| \le \delta } |B^{(N)}_{s'} - B^{(N)}_{s}| \ge \varepsilon \right) \notag\\
	&\quad + \Prob\left( \sup_{0 \le t \le T} |[M^{(N)}]_t - m_t| \ge \delta\right).
\end{align}
By letting $N \to \infty$, we obtain that 
\[
	\limsup_{N \to \infty} \Prob\left(\sup_{0 \le t \le T} |M^{(N)}_t - A^{(N)}_t| \ge \varepsilon \right) \le \Prob\left(\sup_{0 \le s \le m_T} \sup_{s' \ge 0, |s - s'| \le \delta } |B^{(N)}_{s'} - B^{(N)}_{s}| \ge \varepsilon \right).
\]
Note that the right hand side of the above equation does not depend on $N$. It remains to show that it converges to zero as $\delta \to 0+$. Now let $B_t$ be a given standard Brownian motion. Since with probability one, its path is continuous, it follows that as $\delta \to 0+$,
\[
	\sup_{0 \le s \le m_T} \sup_{s' \ge 0, |s - s'| \le \delta } |B_{s'} - B_{s}| \to 0, \text{a.s..} 
\]
Consequently, for any given $\varepsilon > 0$, as $\delta \to 0+$,
\[
	\Prob\left(\sup_{0 \le s \le m_T} \sup_{s' \ge 0, |s - s'| \le \delta } |B^{(N)}_{s'} - B^{(N)}_{s}| \ge \varepsilon \right) = \Prob\left(\sup_{0 \le s \le m_T} \sup_{s' \ge 0, |s - s'| \le \delta } |B_{s'} - B_{s}| \ge \varepsilon \right) \to 0,
\]
which completes the proof.
\end{proof}

\begin{proof}[Proof of Lemma~\rm\ref{lem:joint-convergence-with-initial}]. For each martingale $\Phi^{(f_k; N)}_t$, since $f'_k$ is a non-zero polynomial, it follows that with probability one, 
\[
	[\Phi^{(f_k; N)}]_\infty = \infty.
\]
Recall that 
\[
	[\Phi^{(f_k; N)}]_t \Pto \bra{\nu_c, 2x(1-x)(f_k'(x))^2} \times t, \quad \text{in}\quad C([0, T]). 
\]
Thus, $\Phi^{(f_k; N)}_t$ satisfies the assumptions in Lemma~\ref{lem:time-changed}.
Denote by $A^{(f_k; N)}_t$ the corresponding process introduced in Lemma~\ref{lem:time-changed}. Then as proved in Lemma~\ref{lem:time-changed}, 
\begin{equation}\label{Phi-A}
	\Phi^{(f_k; N)}_t - A^{(f_k; N)}_t \Pto 0, \quad \text{in} \quad C([0, T]).
\end{equation}
Consequently, the processes $\{A^{(f_k; N)}_t\}_{k=1}^n$ jointly converge in distribution to the same limits $\{\bra{\eta, f_k}\}_{k=1}^n$ as $\{\Phi^{(f_k; N)}_t\}_{k=1}^n$ do. In addition, the processes $\{A^{(f_k; N)}_t\}_{k=1}^n$ are independent of the initial data $\{\sqrt{N}(\bra{\mu_0^{(N)}, f_k} - \bra{\nu_c, f_k} ) \}_{k=1}^n$. Thus, the following joint convergence holds,
\[
	\Big\{	\{\sqrt{N}(\bra{\mu_0^{(N)}, f_k} - \bra{\nu_c, f_k} ) \}_{k=1}^n , \{ A^{(f_k; N)}_t\}_{k=1}^n  \Big\} \dto \Big \{\{\bra{\zeta, f_k}\}_{k=1}^n, \{\bra{\eta, f_k}\}_{k=1}^n \Big\}.
\]
Finally, we can replace $A$'s by $\Phi$'s because of \eqref{Phi-A}. The proof is complete.
\end{proof}

Equation~\eqref{Ito-sk}, together with the above result on the joint convergence of initial data and martingale parts implies the following result.
\begin{theorem}\label{thm:moment-processes}
There are Gaussian processes $\{\xi_n\}_{n \ge 1}$ such that for each $n = 1, 2, \dots,$ as $N \to \infty$, 
\[
	\tilde S_n^{(N)} \dto \xi_n,
\]
as $C([0, T])$-valued random elements.
The joint convergence also holds. Consequently, there is a family of Gaussian processes $\bra{\xi, f}$ indexed by polynomials $f$ constructed from $\zeta$'s and $\eta$'s such that for any polynomial $f$
\[
	\sqrt N \left( \bra{\mu_t^{(N)}, f} - \bra{\nu_c, f}\right) \dto \bra{\xi, f}.
\]
\end{theorem}

Theorem~\eqref{thm:moment-processes} can be proved by using induction with the help of Lemma~\ref{lem:joint-processes} below. We skip its proof and the proof of Lemma~\ref{lem:joint-processes} because arguments are almost the same as those used in \cite{NTT-2023}.
\begin{lemma}\label{lem:joint-processes}
Assume that $\tilde S_{n_1}^{(N)}$ and $\tilde S_{n_2}^{(N)}$ jointly converge in distribution to $\xi_{n_1}$ and $\xi_{n_2}$ as $C([0, T])$-valued random elements. Then 
\[
	\sqrt{N}( S_{n_1}^{(N)}  S_{n_2}^{(N)} - u_{n_1} u_{n_2}) \dto u_{n_1} \xi_{n_2} + u_{n_2} \xi_{n_1}.
\]
\end{lemma}

In fact,  the family $\{\sqrt N ( \bra{\mu_t^{(N)}, f} - \bra{\nu_c, f})\}$, martingale parts $\{\Phi^{(f; N)}\}$ and initial data $\{\sqrt N ( \bra{\mu_0^{(N)}, f} - \bra{\nu_c, f} )\}$ jointly converge in distribution to $\xi$'s, $\eta$'s and $\zeta$'s. Let $P$ be a polynomial and $p = P'$. For convenience, we rewrite the formula~\eqref{Ito-f} with $P = f$ as follows
\begin{align}
	d \bra{\mu_t^{(N)}, P} &= \frac{1}{\sqrt N} d \Phi^{(P; N)}(t)  +  \bra{\mu_t^{(N)}, (a + 1) p(x) - (a + b + 2) x p(x) + x(1-x) p'(x)} dt \notag\\
	&\quad + c \iint \frac{x(1-x) p(x) - y(1-y) p(y)}{x - y} d\mu_t^{(N)}(x) d\mu_t^{(N)}(y) dt  \notag\\
	&\quad - \frac cN \bra{\mu_t^{(N)}, \{x(1-x)p(x)\}'} dt . \label{Ito-P}
\end{align}
Since $p$ is a polynomial, it is clear that the integrand of the above double integral can be written as a linear combination of $x^k y^l$, that is,
\[
	\frac{x(1-x) p(x) - y(1-y) p(y)}{x - y} = \sum_{\text{finite}} c_{k,l} x^k y^l.
\]
Thus, that double integral is a linear combination of $\bra{\mu_t^{(N)}, x^k} \bra{\mu_t^{(N)}, x^l} $. Consequently, its limit as $N \to \infty$ exists 
\begin{align*}
&	\iint \frac{x(1-x) p(x) - y(1-y) p(y)}{x - y} d\mu_t^{(N)}(x) d\mu_t^{(N)}(y) \\
&\qquad  \Pto \iint \frac{x(1-x) p(x) - y(1-y) p(y)}{x - y} d\nu_c(x) d\nu_c(y), \quad \text{in $C([0, T])$}.
\end{align*}
Now write equation~\eqref{Ito-P} in the integral form and then let $N \to \infty$, we obtain that
\begin{align*}
	\bra{\nu_c, P}' (=0) &= \bra{\nu_c, (a + 1) p (x)- (a + b + 2) x p (x) + x(1-x) p'(x)} \\
	&\quad + c \iint \frac{x(1-x) p(x) - y(1-y) p(y)}{x - y} d\nu_c(x) d\nu_c(y) .
\end{align*}
Using this identity, again we write the equation~\eqref{Ito-P} in the integral form, subtract each term by the corresponding limit when $N \to \infty$, multiply it by $\sqrt N$ and then let $N \to \infty$, we arrive at the following relation
\begin{equation}\label{xi-zeta-eta}
	\bra{\xi, P}(t) = \bra{\zeta, P} + \bra{\eta, P}(t) + \int_0^t \bralr{\xi, \cL(p)}(s) ds,
\end{equation}
where $\cL$ is a transformation which maps a polynomial $p(x)$ to a polynomial
\begin{align}
	\cL(p)(x) &:= 2c \int \frac{x(1-x)p(x) - y(1-y)p(y)}{x-y} d\nu_c(y) \notag\\
	&\quad + (a+1)p(x) - (a+b+2)xp(x)+x(1-x)p'(x). \label{L-transfromation0}
\end{align}
Here we have used Lemma~\ref{lem:joint-processes} to deal with the double integral term.

It follows from the covariance formula~\eqref{covariance-fg} that when we take primitives $\{P_n\}$ of orthogonal polynomials $\{p_n\}$ with respect to the measure $2x(1-x)\nu_c(x)dx$, the limiting Gaussian processes $\{\bra{\eta, P_n}\}$ are independent.

 In Sect.~\ref{sect:Duals}, the sequence of orthogonal polynomials $\{p_n\}$ (w.r.t.\ $2x(1-x)\nu_c(x)dx$) is taken by using the three term recurrence relation~\eqref{orthogonal-polynomials-pn}.
Let $P_n$ be a primitive of $p_n$. Then Proposition~\ref{lem:L-operator} shows that 
\[
	\cL(p_n) =-\gamma_n  P_n + const, \quad \gamma_n = (n+1)(n+2c+a+b+2).
\]
Thus the relation~\eqref{xi-zeta-eta} becomes
\[
	\bra{\xi, P_n}(t) = \bra{\zeta, P_n}  +  \bra{\eta, P_n}(t) - \gamma_{n} \int_0^t \bralr{\xi, P_n}(s) ds.
\]
Solving this ODE with respect to $\int_0^t \bralr{\xi, P_n}(s) ds$ (see Lemma~\ref{lem:ODE}), we obtain that
\begin{align}
	\bra{\xi, P_n}(t) &= \bra{\zeta, P_n} e^{ -\gamma_{n}t} + \bra{\eta, P_n}(t)  - \gamma_{n}e^{ -\gamma_{n}t} \int_0^t e^{\gamma_n s} \bra{\eta, P_n}(s) ds	\notag\\
	&= \bra{\zeta, P_n} e^{ -\gamma_{n}t}+ e^{-\gamma_n t} \int_0^t e^{\gamma_n s} d\bra{\eta, P_n}(s).\label{xiPn}
\end{align}
Here the last integral is a stochastic integral with respect to the Brownian motion $\bra{\eta, P_n}$. We arrive at the main result in this section.

\begin{theorem}\label{thm:CLT-orthogonal}
The following hold.
\begin{itemize}
\item[\rm(i)] For each $n$, the limiting Gaussian process $\bra{\xi, P_n}$ has the same distribution with the following process
\[
	\zeta e^{-\gamma_n t} +  \alpha_n e^{-\gamma_{n}t} \int_0^t e^{\gamma_{n}s} dw(s), \quad (\alpha_n^2 = \bra{\nu_c, 2x(1-x)p_n(x)^2}),
\]
where $w$ is a standard Brownian motion in $\R$ independent of a Gaussian random variable $\zeta \sim \Normal (0, \frac{\alpha_n^2}{2 \gamma_n})$.

\item[\rm(ii)] The limiting Gaussian processes $\{\bra{\xi, P_n}\}_{n \ge 0}$ are independent. 
 In other words, $\big\{	\sqrt N \big(\bra{\mu_t^{(N)}, P_n} - \bra{\nu_c, P_n}	\big)	\big\}_{n \ge 0}$ jointly converge to independent Gaussian processes.
 \end{itemize}
\end{theorem}
\begin{remark}
Let $\nu_c^*(x)dx = (Z_c^*)^{-1}x(1-x)\nu_c(x)dx$ be a probability measure, where $Z_c^*= \bra{\nu_c, x(1-x)}$ is the normalizing constant. Then $\{p_n\}$ are orthogonal with respect to $\nu_c^*$. Moreover, with notations in Sect.~\ref{sect:Duals}, $\alpha_n^2$ has the following explicit expression
\begin{align*}
	\alpha_n^2 &= \bra{\nu_c, 2x(1-x)p_n(x)^2} = {2}{ \bra{\nu_c, x(1-x)}} \bra{\nu_c^*, p_n(x)^2} \\
	&= 2 \hat \lambda_0(1 - \hat \lambda_0 - \mu_1) \prod_{i=1}^n \lambda^*_{i} \mu^*_{i-1}.
\end{align*}
\end{remark}

\begin{proof}
(i) Recall that the covariance formula~\eqref{covariance-fg} implies that $\bra{\eta, P_n} / \alpha_n$ is a standard Brownian motion. Thus, it suffices to identify the variance of $\bra{\zeta, P_n}$. We first calculate the variance of $\bra{\xi, P_n}(t)$ from the formula~\eqref{xiPn} as
\begin{align*}
	\Ex[(\bra{\xi, P_n}(t))^2] &= \Ex[\bra{\zeta, P_n}^2] e^{-2\gamma_n t} + \alpha_n^2 e^{-2\gamma_{n}t} \int_0^t e^{2 \gamma_{n}s} ds \\
	&=  \Ex[\bra{\zeta, P_n}^2] e^{-2\gamma_n t}  + \frac{\alpha_n^2}{2\gamma_n}- \frac{\alpha_n^2}{2\gamma_n} e^{-2\gamma_n t}.
\end{align*}
In addition the stationary property implies that the variance of $\bra{\xi, P_n}(t)$ does not depend on $t$. It then follows that
\[
	\Ex[\bra{\zeta, P_n}^2] = \frac{\alpha_n^2}{2\gamma_n}.
\]

(ii) We show that $\Cov(\bra{\zeta, P_n}, \bra{\zeta, P_m}) = 0$, for $n \neq m$. Indeed, for $n \neq m$, it follows from the relation~\eqref{xiPn} for $\bra{\xi, P_n}$ and $\bra{\xi, P_m}$ that 
\[
	\Cov(\bra{\xi, P_n}(t), \bra{\xi, P_m}(t)) = \Cov(\bra{\zeta, P_n}, \bra{\zeta, P_m})e^{-(\gamma_n + \gamma_m) t}.
\]
Here we have used the fact that $\bra{\eta, P_m}$ and $\bra{\eta, P_n}$ are independent Gaussian processes independent of $\{\bra{\zeta, P_m}, \bra{\zeta, P_n}\}$. Again, the stationary property implies that 
\[
	\Cov(\bra{\zeta, P_n}, \bra{\zeta, P_m}) = \Cov(\bra{\xi, P_n}(t), \bra{\xi, P_m}(t)) \equiv 0.
\] 
The covariances are zero for $n \neq m$ implies that $\{\bra{\zeta, P_n}\}$ are independent because they are jointly Gaussian random variables. The desired result follows immediately, The proof is complete.
\end{proof}

\section{Associated Jacobi polynomials}\label{sect:Duals}
\subsection{Orthogonal polynomials}
Let us introduce some concepts and results without proofs on orthogonal polynomials. We refer readers to  \cite{Deift-book-1999} or \cite{Simon-book-2011} for proofs and more explanations.. Let $\mu$ be a probability measure on the real line $\R$ having all finite moments, that is, 
\[
	\int x^n d\mu(x) < \infty, \quad \text{for all $n = 1,2,\dots.$}
\]
Assume for instance that $\mu$ is supported on an infinite set. Then the monomials $\{1, x, x^2, \dots\}$ are linearly independent in $L^2(\R, \mu)$. Let 
\[
	p_n = x^n + \text{lower order terms},\quad (n = 0,1,2,\dots)
\]
be the sequence of monic polynomials resulting from the Gram--Schmidt orthogonalization process of the monomials with respect to the usual inner product in $L^2(\R, \mu)$. Then $\{p_n\}_{n \ge 0}$ satisfy the three term recurrence relation
\begin{equation}\label{three-term}
	\begin{cases}
		p_0 = 1, \quad p_1 = x - a_1, \\
		p_{n+1} = xp_n - a_{n+1}p_n - b_{n}^2 p_{n-1}, \quad n \ge 1,
	\end{cases}
\end{equation}
for real numbers $\{a_n\}_{n \ge 1}$ and positive numbers $\{b_n\}_{n \ge 1}$. The (infinite) tridiagonal matrix $J$ formed from the two sequences  $\{a_n\}_{n \ge 1}$ and $\{b_n\}_{n \ge 1}$
\begin{equation}\label{Jacobi-matrix}
	J = \begin{pmatrix}
		a_1		&b_1\\
		b_1		&a_2		&b_2\\
		&\ddots	&\ddots	&\ddots
	\end{pmatrix},
\end{equation}
is called the Jacobi matrix of $\mu$. A symmetric tridiagonal matrix itself is called a Jacobi matrix.

Conversely, given two sequences  $\{a_n\}_{n \ge 1}$ of real numbers and $\{b_n\}_{n \ge 1} $ of positive numbers, or given an (infinite) Jacobi matrix $J$ of the form~\eqref{Jacobi-matrix}, there is a probability measure $\mu$ satisfying 
\begin{equation}\label{spectral-measure}
	\bra{\mu, x^n} = J^{n}(1,1), \quad n = 0,1,2,\dots.
\end{equation}
Then the sequence of polynomials $\{p_n\}_{n\ge 0}$ defined by the three term recurrence relation~\eqref{three-term} are orthogonal with respect to $\mu$  
\[
	\int p_n (x) p_m (x) d\mu(x) = \begin{cases}
		b_1^2 \cdots b_n^2, &n = m,\\
		0, &n \neq m.
	\end{cases}
\]
Consequently, the polynomials $\tilde p_n = p_n/(b_1 \cdots b_n)$ are orthonormal in $L^2(\R, \mu)$. Note that the sequence $\{\tilde p_n\}$ satisfies the following relation
\begin{equation}\label{three-term-on}
	\begin{cases}
		\tilde p_0 = 1, \quad b_{1} \tilde p_1 = x - a_1, \\
		b_{b+1} \tilde p_{n+1} = x\tilde p_n - a_{n+1}\tilde p_n - b_{n} \tilde p_{n-1}, \quad n \ge 1.
	\end{cases}
\end{equation}
A probability measure $\mu$ satisfying the moment condition~\eqref{spectral-measure} exists but may not be unique in general. In case of uniqueness, it is called the spectral measure of the Jacobi matrix $J$. A useful sufficient condition for the uniqueness is given by (see Corollary 3.8.9 in \cite{Simon-book-2011})
\[
	\sum_{n=1}^\infty \frac1{b_n} = \infty.
\]
In particular, the uniqueness holds when the sequence $\{b_n\}_{n \ge 1}$ is bounded.

The finite case is much more straightforward. Let 
\[
	J = \begin{pmatrix}
		a_1		&b_1\\
		b_1		&a_2		&b_2\\
		&\ddots	&\ddots	&\ddots\\
		&&b_{N-1}	&a_N
		
	\end{pmatrix}, \quad (a_i \in \R, b_i > 0),
\]
be a Jacobi matrix of size $N$. Then $J$ has $N$ distinct eigenvalues $\lambda_1,\dots, \lambda_N$. Let $v_1, \dots, v_N$ be the corresponding normalized eigenvectors. In this case, the spectral measure $\mu$ of $J$ which satisfies the moment relation~\eqref{spectral-measure} has the following form 
\[
	\mu = \sum_{i=1}^N |v_i(1)|^2 \delta_{\lambda_i}.
\]
The three term relation is defined up to $N$, 
\[
	\begin{cases}
		p_0 = 1, \quad p_1 = x - a_1, \\
		p_{n+1} = xp_n - a_{n+1}p_n - b_{n}^2 p_{n-1}, \quad n = 1,\dots, N-1,
	\end{cases}
\]
with $\{p_n\}_{n=0}^{N-1}$ orthogonal polynomials and $p_N(x) = \prod_{i=1}^N (x - \lambda_i) = 0$ in $L^2(\R, \mu)$.

\subsection{Beta Jacobi ensembles and duals of Jacobi polynomials}
In this subsection, we consider the limiting behavior of beta Jacobi ensembles in a low temperature regime, that is, the regime where $N$ is fixed, $(\beta, a, b)  = (2 \kappa, A \kappa, B\kappa)$ with $\kappa \to \infty$, for fixed $A, B > 0$. Recall that $L_{N, \beta}$ and $\Sp_{N, \beta}$ denote the empirical distribution and the spectral measure of the random tridiagonal matrix $J_{N, \beta}$.

In this regime, it follows from the asymptotic behavior of beta distributions that 
\[
J_{N, \beta} \to T_N = 
\begin{pmatrix}
		\sqrt{c_1}	\\
		\sqrt{d_1}	&\sqrt{c_2}		\\
		&\ddots	&\ddots \\
		&& \sqrt{d_{N - 1}}	& \sqrt{c_N}
	\end{pmatrix}
	 \begin{pmatrix}
		\sqrt{c_1} 	&\sqrt{d_1}	\\
			&\sqrt{c_2}		&\sqrt{d_2}		\\
		&&\ddots	&\ddots \\
		&&& \sqrt{c_N}
	\end{pmatrix}, 	
\]
where 
\begin{align*}
	c_1 &= \frac{1 - N - A}{2 - 2N - A - B},	\\
	c_n &= \frac{n - N - A}{2n - 2N - A - B} \frac{n - N - A - B}{2n - 2N - A - B - 1}, \quad n \ge 2,\\
	d_n &=  \frac{n - N}{2n - 2N - A - B + 1} \frac{n - N - B}{2n - 2N - A - B}, \quad n \ge 1.
\end{align*}
Namely, each entry of $J_{N, \beta}$ converges in probability and in $L^q$, for any $q \in [1, \infty)$, to the corresponding entry of $T_N$. Since the size $N$ is fixed, it follows that for any $k = 1, 2, \dots,$
\begin{equation}\label{convergence-moments}
	\frac{1}{N} \sum_{i=1}^N (J_{N, \beta})^k(i,i) \to \frac1N \sum_{i=1}^N (T_N)^k(i, i), \quad (J_{N, \beta})^k(1,1) \to (T_N)^k(1, 1), 
\end{equation}
in probability and in $L^q$ for any $q \in [1, \infty)$. Recall from the identity~\eqref{same-mean} that the two random variables in the above expression have the same mean value.
Consequently, the two limits are identical
\[
	 \frac1N \sum_{i=1}^N (T_N)^k(i, i) = (T_N)^k(1, 1), \quad \text{for all $k=1, 2, \dots$.}
\]
Let $z_1 < z_2 < \cdots < z_N$ be the eigenvalues of $T_N$. The above relation implies that the spectral measure $\mu_N$ of $T_N$ coincides with the empirical measure $N^{-1} \sum_{i=1}^N \delta_{z_i}$. Here the spectral measure $\mu_N$ is the unique probability measure satisfying 
\[
	\bra{\mu_N, x^k} = (T_N)^k(1,1), \quad k = 0,1,2,\dots.
\]
Then the convergence of moments~\eqref{convergence-moments} implies that both $L_{N, \beta}$ and $\Sp_{N, \beta}$ converge weakly to $\mu_N$, in probability.

Beta Jacobi ensembles at low temperature can also be studied by analyzing the joint density \cite{Hermann-Voit-2021}. Note that beta Jacobi ensembles considered in \cite{Hermann-Voit-2021} are ensembles of points in the interval $[-1, 1]$. Their results will be translated to fit our setting as follows.  We rewrite the joint density in the following form
\begin{align}
	&\frac 1Z \times \prod_{i < j} |\lambda_j - \lambda_i|^{2\kappa} \prod_{l = 1}^N \lambda_l^{A\kappa} (1 - \lambda_l)^{B\kappa}	\notag\\
	&=	\frac 1Z \times \bigg( \prod_{i < j} (\lambda_j - \lambda_i) \prod_{l = 1}^N \lambda_l^{\frac A2} (1 - \lambda_l)^{\frac B2}\bigg)^{2\kappa}, \quad ( \lambda_1, \dots,  \lambda_N) \in \bar \cWA 	\label{bJE-low},
\end{align}
where $\bar \cWA$ is the closure of 
\[
	\cWA = \{(x_1, \dots, x_N) \in \R^N : 0 < x_1 < x_2 < \cdots < x_N < 1\}.
\]
Let the function $\phi \colon \cWA \to \R$ be defined as
\[
	\phi(\lambda_1, \dots, \lambda_N) =  \prod_{i < j} (\lambda_j - \lambda_i) \prod_{l = 1}^N \lambda_l^{\frac A2} (1 - \lambda_l)^{\frac B2}, \quad  (\lambda_1, \lambda_2,  \dots, \lambda_N) \in \bar \cWA.
\]
For $A, B > 0$, one can show that $\phi$ has a unique maximum in $\bar \cWA$ at an inner point $(y_1, \dots, y_N) \in \cWA$. Let $X_\kappa = (\lambda_1, \dots, \lambda_N)$ be a  random vector distributed according to the beta Jacobi ensemble~\eqref{bJE-low}. Then by a well-known Laplace method (see \cite[Lemma~2.1]{Hermann-Voit-2021}), as $\kappa \to \infty$,  
\[
	X_\kappa \Pto (y_1, \dots, y_N).
\]
Consequently, the empirical distribution $L_{N, \beta}$ converges weakly to $\frac1N \sum_{i=1}^N \delta_{y_i}$, in probability. As a byproduct, we deduce that $(z_1, \dots, z_N) = (y_1, \dots, y_N)$.

In addition, from a classical theory of Jacobi polynomials, the points $\{y_1, \dots, y_N\}$ are the zeros of the $N$th Jacobi polynomials $P_N^{(A-1, B-1)}$ with parameters $A-1$ and $B-1$. Here Jacobi polynomials $\{P_n^{(A-1, B-1)}\}_{n \ge 0}$ are monic polynomials orthogonal with respect to the beta distribution with parameters $A,B>0$, or the probability measure with density proportional to $x^{A-1}(1-x)^{B-1}, x\in(0,1)$. It is also known that the Jacobi matrix of the beta distribution with parameters $A, B > 0$ is given by
\begin{align*}
	H &= 	\begin{pmatrix}
		\sqrt{\lambda_0}	\\
		\sqrt{\mu_1}	&\sqrt{\lambda_1}	\\
		&\ddots	&\ddots	
	\end{pmatrix}
	\begin{pmatrix}
		\sqrt{\lambda_0}	&\sqrt{\mu_1}	\\
			&\sqrt{\lambda_1}	&\sqrt{\mu_2}	\\
			&&\ddots	&\ddots	
	\end{pmatrix}
	=\begin{pmatrix}
		a_1	&b_1\\
		b_1	&a_2		&b_2\\
		&\ddots	&\ddots	&\ddots
	\end{pmatrix},\\
	&\qquad
	\begin{cases}
	\lambda_n = \frac{n + A}{2n + A+B}\frac{n + A+B-1}{2n+ A+B-1}, &n \ge 0,\\
	\mu_n =  \frac{n}{2n + A+B-1}\frac{n  + B-1}{2n + A+B-2}, & n \ge 1,\\
	a_1 = \lambda_0,	&\\
	a_{n} = \lambda_{n} + \mu_n, &n \ge 2,\\
	b_n^2 = \lambda_{n-1} \mu_n, &n \ge 1.
	\end{cases}
\end{align*}
Using such Jacobi matrix, Jacobi polynomials $\{P^{(A-1, B-1)}_n\}_{n \ge 0}$ can be defined by
\begin{align*}
	&P^{(A-1, B-1)}_0 = 1,\quad P^{(A-1, B-1)}_1 = x - a_1,\\
	&P^{(A-1, B-1)}_{n+1} = x P^{(A-1, B-1)}_n - a_{n+1}P^{(A-1, B-1)}_n - b_n^2 P^{(A-1, B-1)}_{n-1}, \quad n \ge 1.
\end{align*}

Let $H_{[N]}$ be the top left submatrix of size $N \times N$ of the semi-infinite matrix $H$,
\[
H_{[N]} = \begin{pmatrix}
	a_1	&b_1\\
	b_1	&a_2		&b_2\\
	&\ddots	&\ddots	&\ddots\\
	&&b_{N-1}	&a_N
\end{pmatrix}.
\] 
Then it follows from the three term recurrence relation that 
\[
	P_{N}^{(A-1, B-1)}(x) = \det(x - H_{[N]}).
\]
Thus, the zeros $(y_1, \dots, y_N)$ of $P_{N}^{(A-1, B-1)}$ are the eigenvalues of $H_{[N]}$.

Now let $T^*_N$ be the  matrix obtained by reversing the parameters in $H_{[N]}$, that is, 
\[
	T^*_N  = \begin{pmatrix}
	a_1^*	&b_1^*\\
	b_1^*	&a_2^*		&b_2^*\\
	&\ddots	&\ddots	&\ddots\\
	&&b_{N-1}^*	&a_N^*
\end{pmatrix}, \quad a_i^* = a_{N-i+1}, b_i^* = b_{N-i}. 
\]
We define orthogonal polynomials $\{Q_{n}\}_{n=0}^{N-1}$ with respect to $T^*_N$ as
\begin{align*}
	&Q_0 = 1,\quad Q_1 = x - a_N,\\
	&Q_{n+1} = x Q_n - a_{N-n} Q_n - b_{N-n}^2 Q_{n-1}, \quad n = 1, \dots, N -2.
\end{align*}
The polynomials $\{Q_{n}\}_{n=0}^{N-1}$ are called the duals of Jacobi polynomials $\{P^{(A-1, B-1)}_n\}_{n=0}^{N-1}$. On the one hand, they are orthogonal with respect to the spectral measure of $T_N^*$. On the other hand, they are orthogonal w.r.t.\ the probability measure
\[
	\mu^*_N =\frac{1}{Z^*} \sum_{i=1}^N y_{i}(1 - y_{i}) \delta_{y_{i}}.
\]
where $Z^*$ is a normalizing constant \cite{Vinet-Zhedanov-2004} (see also \cite{Andraus-2020}). Thus, the measure $\mu^*_N$ is the spectral measure of $T_N^*$. Then we derive the following relation between $T_N$ and $T_N^*$.
\begin{lemma}\label{lem:dual-of-JN}
	For any $n = 0,1,2, \dots,$ it holds that 
	\begin{equation}\label{TN}
		(T^*_N)^n(1,1) = \frac{(T_N)^{n+1}(1,1) - (T_N)^{n+2}(1,1)}{T_N(1,1) - (T_N)^2(1,1)}.
	\end{equation}
\end{lemma}
\begin{proof}
Recall that the measure 
\[
	\mu_N = \frac1N \sum_{i=1}^N \delta_{z_i}
\]
is the spectral measure of $T_N$, and the measure 
\[
	\mu_N^* =\frac{1}{Z^*} \sum_{i=1}^N z_{i}(1 - z_{i}) \delta_{z_{i}}
\] 
is the spectral measure of $T_N^*$. (Because $(y_1, \dots, y_N) = (z_1, \dots, z_N)$.) The normalizing constant $Z^*$ is given by
\[
	Z^* = \sum_{i=1}^N z_i(1 - z_i) = N \bra{\mu_N, x - x^2}.
\]
From the definition of spectral measures, we see that for any $n = 0,1,\dots,$,
\[
	(T_N^*)^n(1,1) = \bra{\mu_N^*, x^n} = \frac1{Z^*} \sum_{i=1}^N z_i(1-z_i) z_i^n = \frac{\bra{\mu_N, x^{n+1} - x^{n+2}}}{\bra{\mu_N, x-x^2}},
\]
and any $k = 0, 1, \dots$,
\[
	\bra{\mu_N, x^k} = \frac1N \sum_{i=1}^N z_i^k = (T_N)^k(1,1).
\]
From those, the relation~\eqref{TN} immediately follows. The proof is complete.
\end{proof}

\subsection{The limiting measure in a high temperature regime}
We first recall arguments in \cite{Trinh-Trinh-Jacobi} which were used to derive the Jacobi matrix of the limiting measure $\nu_c$ in the regime where $\beta N \to 2c \in (0, \infty)$.
Let $m_k(N, \kappa, a, b)$ be the expected value of the $k$th moment of the empirical distribution $L_{N, \beta}$, that is, 
\[
	m_k(N, \kappa, a, b) = \Ex[\bra{L_N, x^k}] = \Ex\left[\frac1N \tr((J_{N, 2\kappa})^k)\right] =\Ex [(J_{N, 2\kappa})^k(1,1)].
\]
Here $\kappa = \beta/2$. 
Then from formulas for moments of the beta distribution, we see that $m_k$ can be defined for any $N, \kappa, a$ and $b$. The function $m_k$ satisfies the following duality relation
\begin{equation}\label{moment-duality}
		m_k(N, \kappa, a, b) =  m_k(-\kappa N, 1/\kappa, - a/\kappa, -b/\kappa),
\end{equation}
(see \cite[Appendix A]{Dumitriu-Paquette-2012} and \cite[Eq.\ (4.15)]{Forrester-2017}).

Next, we form an infinite matrix $J_c$ by changing $N \leftrightarrow -c, A \leftrightarrow -a, B \leftrightarrow -b$ in the matrix $T_N$, that is, 
\[
	J_c =\begin{pmatrix}
		\hat\lambda_0 	& \sqrt{\hat\lambda_0 \mu_1}	\\
		 \sqrt{\hat\lambda_0 \mu_1}	&\lambda_1 + \mu_1		& \sqrt{\lambda_1 \mu_2}	\\
		 &\ddots	&\ddots	&\ddots	
	\end{pmatrix},
\]
where
\begin{align*}
	\hat\lambda_0 &=   \frac{c + a + 1}{2c + a + b + 2},\\
	\lambda_n &=   \frac{n + c + a + 1}{2n + 2c + a + b + 2} \frac{n + c + a + b + 1}{2n + 2c + a + b + 1}, \quad n \ge 0,\\
	\mu_n &=   \frac{n + c}{2n + 2c + a + b + 1} \frac{n + c + b}{2n + 2c + a + b}, \quad n \ge 0.
\end{align*}
Then the duality relation implies that the limit of the empirical distribution $L_{N, \beta}$ in the regime where $\beta N \to 2c$ coincides with the spectral measure of $J_c$. In other words, the limiting measure $\nu_c$ is the spectral measure of $J_c$.
The density of $\nu_c$ was explicitly calculated in \cite{Trinh-Trinh-Jacobi}. It is viewed as Model III of associated Jacobi polynomials.

In the same manner, we now define an infinite Jacobi matrix $J_c^*$ by changing $N \leftrightarrow -c, A \leftrightarrow -a, B \leftrightarrow -b$ in the matrix $T_N^*$,  that is, 
\begin{equation}\label{J*}
	J^*_c=\begin{pmatrix}
		\lambda_0^* + \mu_0^*	&\sqrt{\mu_0^* \lambda_1^*}\\
		\sqrt{\mu_0^* \lambda_1^*}	&\lambda_1^* + \mu_1^*	&\sqrt{\mu_1^* \lambda_2^*}\\
		&\ddots	&\ddots	&\ddots
	\end{pmatrix},
\end{equation}
where
\begin{align}
	\lambda_n^* & = \frac{n+c+a+1}{2n+2c+a+b+2}\frac{n+c+a+b+2}{2n+2c+a+b+3},\\
	\mu_n^* &= \frac{n+c+1}{2n+2c+a+b+3}\frac{n+c+b+2}{2n+2c+a+b+4}, \quad n \ge 0.
\end{align}
\begin{lemma}
The matrix $J^*_c$ is the Jacobi matrix of the probability measure 
\begin{equation}\label{nu*}
	\nu^*_c=(Z_c^*)^{-1}x(1-x)\nu_c(x)dx, \quad \text{with}\quad Z_c^* =  \bra{\nu_c, x(1-x)} = \hat \lambda_0(1 - \hat \lambda_0 - \mu_1).
\end{equation}
\end{lemma}
\begin{proof}
The relation~\eqref{TN} in Lemma~\ref{lem:dual-of-JN} implies that for any $n=0,1,\dots,$
\[
		(J^*_c)^n(1,1) = \frac{(J_c)^{n+1}(1,1) - (J_c)^{n+2}(1,1)}{J_c(1,1) - (J_c)^2(1,1)}.
\]
Thus, the measure $\nu^*_c$ defined in \eqref{nu*} satisfies 
\[
	\bra{\nu^*_c, x^n} = \frac{\bra{\nu_c, x(1-x)x^n}}{\bra{\nu_c, x(1-x)}} = \frac{(J_c)^{n+1}(1,1) - (J_c)^{n+2}(1,1)}{J_c(1,1) - (J_c)^2(1,1)} = (J^*_c)^n(1,1), \quad n = 0,1,\dots.
\]
Therefore, $\nu^*_c$ is the spectral measure of $J_c^*$. The lemma is proved.
\end{proof}

\begin{remark}
	By a direct calculation, we see that the Jacobi matrix $J_c^*$ coincides with the Jacobi matrix of Model I of associated Jacobi polynomials (see the matrix $J_I$ in \cite[\S 4]{Trinh-Trinh-Jacobi}) with parameters $(c,a+1,b+1)$. This fact is similar to what happens in the Laguerre case (see \cite{NTT-2023}).
\end{remark}

Recall that the transformation $\cL$ is defined on polynomials as
\begin{align}
	\cL(p)(x) &:= 2c \int \frac{x(1-x)p(x) - y(1-y)p(y)}{x-y} d\nu_c(y) \notag\\
	&\quad + (a+1)p(x) - (a+b+2)xp(x)+x(1-x)p'(x). \label{L-transfromation}
\end{align}
Let $\{p_n\}_{n \ge 0}$ be orthogonal polynomials with respect to $\nu_c^*$ which are defined by using the Jacobi matrix $J^*_c$ as
\begin{equation}\label{orthogonal-polynomials-pn}
\begin{cases}
	p_0(x) = 1, \quad 	p_1(x) = x - a_1^*,\\
	p_{n+1}(x) = x p_n(x)  - a_{n+1}^* p_n(x) - (b_n^*)^2p_{n-1}(x), \quad n = 1,2,\dots,
\end{cases}
\end{equation}
where $a_n^* = \lambda_{n-1}^* + \mu_{n-1}^*$ and $b_n^* = \sqrt{\mu_{n-1}^* \lambda_n^*}$.
Results on the Gaussian case and the Laguerre case suggest the following.
\begin{proposition}\label{lem:L-operator}
It holds that
\begin{equation}\label{L-pn}
	\cL(p_n) = - \gamma_n P_n + const,\quad \gamma_n = (n+1)(n+2c+a+b+2),
\end{equation}
where $P_n$ is a primitive of $p_n$, that is, $P_n' = p_n$.
\end{proposition}
\begin{proof}
Define 
\[
	q_n(x) = \int\frac{p_n(x) - p_n(y)}{x-y} d\nu^*_c(y).
\]
Then $\{q_n\}$ satisfy the same three term recurrence relation as $\{p_n\}$ but with different initial conditions ($q_0 = 0, q_1 = 1$). We observe that
\begin{align*}
	&2c \int \frac{x(1-x)p_n(x) - y(1-y)p_n(y)}{x-y} d\nu_c(y) \\
	&=2c \int \frac{(x(1-x)-y(1-y))p_n(x) + y(1-y)(p_n(x)-p_n(y))}{x-y} d\nu_c(y) \\
	&=2c(1 - x - \hat \lambda_0) p_n(x) + 2c \hat \lambda_0(1- \hat \lambda_0 - \mu_1)q_n(x).
\end{align*}
Using this, the desired identity~\eqref{L-pn} is equivalent to the following by taking the derivative 
\begin{align*}
	&(2c(1-\hat \lambda_0) + a + 2 )p_n' + 2c \hat \lambda_0(1-\hat \lambda_0-\mu_1)q_n' - (2c+a+b+4)xp_n' + x(1-x)p_n''\\
	&\quad ((n+1)(n+2c+a+b+2) - (2c+a+b+2))p_n = 0.
\end{align*}
We claim that this relation can be shown by induction using recurrence relations for $\{p_n\}$ and $\{q_n\}$. The arguments are direct but rather long, so that we omit the detail.
\end{proof}

\section{Proof of the main result}
Theorem~\ref{thm:CLT-intro} is in fact a direct consequence of Theorem~\ref{thm:CLT-orthogonal}. Note that Theorem~\ref{thm:CLT-orthogonal} states results on orthogonal polynomials $\{p_n\}$ while Theorem~\ref{thm:CLT-intro} states results on orthonormal polynomials $\{\tilde p_n\}$. Theorem~\ref{thm:CLT-orthogonal} implies that 
\[
	\sqrt{N}\Big(\bra{L_{N, \beta}, P_n} - \bra{\nu_c, P_n} \Big) \dto \Normal(0, \sigma_{P_n}^2),
\]
where
\[
	\sigma_{P_n}^2 = \frac{  \bra{\nu_c, x(1-x)p_n^2}}{(n+1)(n+2c+a+b+2) },
\]
because $L_{N, \beta}$ and $\mu_0^{(N)}$ have the same distribution. From  the relation between $\nu_c$ and $\nu_c^*$, we see that 
\[
		 \bra{\nu_c, x(1-x)p_n^2} = \bra{\nu_c, x(1-x)} \bra{\nu_c^*, p_n^2} = \hat \lambda_0(1 - \hat \lambda_0 - \mu_1) \bra{\nu_c^*, p_n^2}.
\]
Note that $\tilde p_n = p_n / (\bra{\nu_c^*, p_n^2})^{1/2}$. Thus, if we take $\tilde P_n = P_n /  (\bra{\nu_c^*, p_n^2})^{1/2}$, we get that 
\[
	\sqrt{N}\Big(\bra{L_{N, \beta}, \tilde P_n} - \bra{\nu_c, \tilde P_n} \Big) \dto \Normal(0, \sigma_{\tilde P_n}^2), \quad  \sigma_{\tilde P_n}^2= \frac{\hat \lambda_0(1 - \hat \lambda_0 - \mu_1)}{(n+1)(n+2c+a+b+2) }.
\]
The first statement in Theorem~\ref{thm:CLT-intro} follows immediately by substituting the formulae of $\hat \lambda_0$ and $\mu_1$. The second statement in Theorem~\ref{thm:CLT-intro} was already included in the proof of Theorem~\ref{thm:CLT-orthogonal}. The proof is complete. \qed

\bigskip
\noindent \textbf{Acknowledgements.} This work is supported by JSPS KAKENHI Grant Numbers 20K03659 (F.N.) and JP19K14547 (K.D.T.). Trinh Hoang Dung, ID VNU.2021.NCS.11, thanks The Development Foundation of Vietnam National University, Hanoi for sponsoring this research.

\bigskip
\begin{footnotesize}


\begin{thebibliography}{10}
\providecommand{\url}[1]{{#1}}
\providecommand{\urlprefix}{URL }
\expandafter\ifx\csname urlstyle\endcsname\relax
  \providecommand{\doi}[1]{DOI~\discretionary{}{}{}#1}\else
  \providecommand{\doi}{DOI~\discretionary{}{}{}\begingroup
  \urlstyle{rm}\Url}\fi

\bibitem{AGS-book}
Ambrosio, L., Gigli, N., Savar\'{e}, G.: Gradient flows in metric spaces and in
  the space of probability measures, second edn.
\newblock Lectures in Mathematics ETH Z\"{u}rich. Birkh\"{a}user Verlag, Basel
  (2008)

\bibitem{Andraus-2020}
Andraus, S., Hermann, K., Voit, M.: Limit theorems and soft edge of freezing
  random matrix models via dual orthogonal polynomials.
\newblock J. Math. Phys. \textbf{62}(8), Paper No. 083,303, 26 (2021).
\newblock \doi{10.1063/5.0028706}.
\newblock \urlprefix\url{https://doi.org/10.1063/5.0028706}

\bibitem{Cepa-1995}
C{\'{e}}pa, E.: \'{E}quations diff\'{e}rentielles stochastiques multivoques.
\newblock In: S\'{e}minaire de {P}robabilit\'{e}s, {XXIX}, \emph{Lecture Notes
  in Math.}, vol. 1613, pp. 86--107. Springer, Berlin (1995).
\newblock \doi{10.1007/BFb0094202}.
\newblock \urlprefix\url{https://doi.org/10.1007/BFb0094202}

\bibitem{Cepa-Lepingle-1997}
C\'{e}pa, E., L\'{e}pingle, D.: Diffusing particles with electrostatic
  repulsion.
\newblock Probab. Theory Related Fields \textbf{107}(4), 429--449 (1997).
\newblock \doi{10.1007/s004400050092}.
\newblock \urlprefix\url{https://doi.org/10.1007/s004400050092}

\bibitem{Deift-book-1999}
Deift, P.A.: Orthogonal polynomials and random matrices: a {R}iemann-{H}ilbert
  approach, \emph{Courant Lecture Notes in Mathematics}, vol.~3.
\newblock New York University, Courant Institute of Mathematical Sciences, New
  York; American Mathematical Society, Providence, RI (1999)

\bibitem{Demni-2010}
Demni, N.: {$\beta$}-{J}acobi processes.
\newblock Adv. Pure Appl. Math. \textbf{1}(3), 325--344 (2010).
\newblock \doi{10.1515/APAM.2010.019}.
\newblock \urlprefix\url{https://doi.org/10.1515/APAM.2010.019}

\bibitem{Dumitriu-Paquette-2012}
Dumitriu, I., Paquette, E.: Global fluctuations for linear statistics of
  {$\beta$}-{J}acobi ensembles.
\newblock Random Matrices Theory Appl. \textbf{1}(4), 1250,013, 60 (2012).
\newblock \urlprefix\url{http://dx.doi.org/10.1142/S201032631250013X}

\bibitem{Forrester-book}
Forrester, P.J.: Log-gases and random matrices, \emph{London Mathematical
  Society Monographs Series}, vol.~34.
\newblock Princeton University Press, Princeton, NJ (2010).
\newblock \doi{10.1515/9781400835416}.
\newblock \urlprefix\url{https://doi.org/10.1515/9781400835416}

\bibitem{Forrester-2021}
Forrester, P.J., Mazzuca, G.: The classical {$\beta$}-ensembles with {$\beta$}
  proportional to {$1/N$}: from loop equations to {D}yson's disordered chain.
\newblock J. Math. Phys. \textbf{62}(7), Paper No. 073,505, 22 (2021).
\newblock \doi{10.1063/5.0048481}.
\newblock \urlprefix\url{https://doi.org/10.1063/5.0048481}

\bibitem{Forrester-2017}
Forrester, P.J., Rahman, A.A., Witte, N.S.: Large {$N$} expansions for the
  {L}aguerre and {J}acobi {$\beta$}-ensembles from the loop equations.
\newblock J. Math. Phys. \textbf{58}(11), 113,303, 25 (2017).
\newblock \urlprefix\url{https://doi.org/10.1063/1.4997778}

\bibitem{Hermann-Voit-2021}
Hermann, K., Voit, M.: Limit theorems for {J}acobi ensembles with large
  parameters.
\newblock Tunis. J. Math. \textbf{3}(4), 843--860 (2021).
\newblock \doi{10.2140/tunis.2021.3.843}.
\newblock \urlprefix\url{https://doi.org/10.2140/tunis.2021.3.843}

\bibitem{Killip-Nenciu-2004}
Killip, R., Nenciu, I.: Matrix models for circular ensembles.
\newblock Int. Math. Res. Not. (50), 2665--2701 (2004).
\newblock \doi{10.1155/S1073792804141597}.
\newblock \urlprefix\url{https://doi.org/10.1155/S1073792804141597}

\bibitem{NTT-2023}
Nakano, F., Trinh, H.D., Trinh, K.D.: Limit theorems for moment processes of
  beta {D}yson's {B}rownian motions and beta {L}aguerre processes.
\newblock Random Matrices: Theory and Applications (to appear)  (2023).
\newblock \doi{10.1142/S2010326323500053}.
\newblock \urlprefix\url{https://doi.org/10.1142/S2010326323500053}

\bibitem{Nakano-Trinh-2018}
Nakano, F., Trinh, K.D.: Gaussian beta ensembles at high temperature:
  eigenvalue fluctuations and bulk statistics.
\newblock J. Stat. Phys. \textbf{173}(2), 295--321 (2018).
\newblock \doi{10.1007/s10955-018-2131-9}.
\newblock \urlprefix\url{https://doi.org/10.1007/s10955-018-2131-9}

\bibitem{Rebolledo-1980}
Rebolledo, R.: Central limit theorems for local martingales.
\newblock Z. Wahrsch. Verw. Gebiete \textbf{51}(3), 269--286 (1980).
\newblock \doi{10.1007/BF00587353}.
\newblock \urlprefix\url{https://doi.org/10.1007/BF00587353}

\bibitem{Revuz-Yor-book}
Revuz, D., Yor, M.: Continuous martingales and {B}rownian motion,
  \emph{Grundlehren der Mathematischen Wissenschaften [Fundamental Principles
  of Mathematical Sciences]}, vol. 293, second edn.
\newblock Springer-Verlag, Berlin (1994)

\bibitem{Simon-book-2011}
Simon, B.: Szeg{\H o}'s theorem and its descendants.
\newblock M. B. Porter Lectures. Princeton University Press, Princeton, NJ
  (2011).
\newblock Spectral theory for $L{\sp{2}}$ perturbations of orthogonal
  polynomials

\bibitem{Trinh-Trinh-Jacobi}
Trinh, H.D., Trinh, K.D.: Beta {J}acobi ensembles and associated {J}acobi
  polynomials.
\newblock J. Stat. Phys. \textbf{185}(1), Paper No. 4, 15 (2021).
\newblock \doi{10.1007/s10955-021-02832-z}.
\newblock \urlprefix\url{https://doi.org/10.1007/s10955-021-02832-z}

\bibitem{Vinet-Zhedanov-2004}
Vinet, L., Zhedanov, A.: A characterization of classical and semiclassical
  orthogonal polynomials from their dual polynomials.
\newblock J. Comput. Appl. Math. \textbf{172}(1), 41--48 (2004).
\newblock \doi{10.1016/j.cam.2004.01.031}.
\newblock \urlprefix\url{https://doi.org/10.1016/j.cam.2004.01.031}

\end{thebibliography}
\end{footnotesize}
\end{document}